\documentclass[graybox]{svmult}

\usepackage{mathptmx}       
\usepackage{helvet}         
\usepackage{courier}        
\usepackage{type1cm}        
\usepackage{makeidx}         
\usepackage{graphicx}        
\usepackage{multicol}        
\usepackage[bottom]{footmisc}

\makeindex             
                       
                       \usepackage{amsmath,mathrsfs,enumitem,tikz}
\usepackage{eepic,pstricks}
\usepackage{graphicx}
\usepackage{footnote,color}

\begin{document}

\title*{$t$-Martin boundary of killed random walks in the quadrant}
\author{C\'edric Lecouvey and Kilian Raschel}
\institute{C\'edric Lecouvey \at Laboratoire de Math\'ematiques et Physique Th\'eorique, Universit\'e de Tours, Parc de Grandmont, 37200 Tours, France. \email{Cedric.Lecouvey@lmpt.univ-tours.fr}
\and Kilian Raschel \at CNRS \& Pacific Institute for the Mathematical Sciences, Vancouver, British Columbia,
Canada \& Laboratoire de Math\'ematiques et Physique Th\'eorique, Universit\'e de Tours, Parc de Grandmont, 37200 Tours, France. \email{Kilian.Raschel@lmpt.univ-tours.fr}}

%
%
\maketitle

\abstract*{We compute the $t$-Martin boundary of two-dimensional small steps random walks killed at the boundary of the quarter plane. We further provide explicit expressions for the (generating functions of the) discrete $t$-harmonic functions. Our approach is uniform in $t$, and shows that there are three regimes for the Martin boundary.}

\abstract{We compute the $t$-Martin boundary of two-dimensional small steps random walks killed at the boundary of the quarter plane. We further provide explicit expressions for the (generating functions of the) discrete $t$-harmonic functions. Our approach is uniform in $t$, and shows that there are three regimes for the Martin boundary.}

\section{Introduction and main results}
\label{sec:Introduction}

\subsection*{Aim of this paper}
This work is concerned with discrete $t$-harmonic functions associated to Laplacian operators with Dirichlet conditions in the quarter plane. Let $\{p_{k,\ell}\}$ be non-negative numbers summing to $1$. Consider the associated discrete $t$-Laplacian, acting on functions $f$ defined on the quarter plane ${\bf N}^2=\{0,1,2,\ldots \}^2$ by
\begin{equation*}
     L_{t}(f)(i,j)= \sum_{k,\ell} p_{k,\ell} f(i+k,j+\ell)-t\cdot f(i,j),\qquad \forall i,j\geq 1.
\end{equation*}
Our aim is to characterize the functions $f=\{f(i,j)\}_{i,j\geq 0}$ which are
\begin{enumerate}
\item\label{prop1}$t$-harmonic in the interior of the quarter plane, i.e., $L_{t}(f)(i,j)=0$ for all $i,j\geq 1$;
\item\label{prop2}positive in the interior of the quarter plane: for all $i,j\geq 1$, $f(i,j)>0$;
\item\label{prop3}zero on the boundary and at the exterior of the quarter plane: for all $(i,j)\in{\bf Z}^2$ such that $i\leq0$ and/or $j\leq0$, $f(i,j)=0$.
\end{enumerate}
The probabilistic counterpart of this potential theory viewpoint is the following: $t$-harmonic functions satisfying to~\ref{prop1}--\ref{prop3} are $t$-harmonic for random walks (whose increments have the law $\{p_{k,\ell}\}$) killed at the boundary of the quarter plane.

\subsection*{Literature}


In general, it is a difficult problem to determine the Martin boundary (essentially, the set of harmonic functions) of a given class of Markov chains, especially for non-homogeneous processes (in our case, the inhomogeneity comes from the boundary of the quadrant). The $t$-Martin boundary plays a crucial role to determine the Martin boundary (i.e., the $t$-Martin boundary with $t=1$) of products of transition kernels \cite{PW,IR10}. Moreover, via the procedure of Doob $h$-transform, discrete harmonic functions have also applications to defining random processes conditioned on staying in given domains of ${\bf Z}^d$ (the latter processes arise a great interest in the mathematical community, as they appear in several distinct domains: quantum random walks \cite{Bi1,Bi3}, random matrices, non-colliding random walks \cite{LPP13}). Further details and motivations of considering the $t$-Martin boundary can be found in \cite[Introduction]{IR10}.

For non-zero drift random walks in cones, the Martin boundary has essentially been found for very particular cones, as half spaces ${\bf N}\times {\bf Z}^{d-1}$ and orthants ${\bf N}^d$. In \cite[Corollary 1.1]{IL10} it has been found in the case of ${\bf N}^2$ for random walks with exponential moments, using ratio limit theorems for local processes and large deviation techniques. The Martin boundary was proved to be homeomorphic to $[0,\pi/2]$. In \cite{KR11}, under the small steps and non-degeneration hypotheses, namely,
\begin{enumerate}[label={\rm (\alph{*})},ref={\rm (\alph{*})}]
\item\label{prop4}the $p_{k,\ell}$ are $0$ as soon as $\vert k\vert >1$ and/or $\vert \ell\vert >1$,
\item\label{prop5}in the (clockwise) list $p_{1,1},p_{1,0},p_{1,-1},p_{0,-1},p_{-1,-1},p_{-1,0},p_{-1,1},p_{0,1}$,
                 there are no three~consecutive zeros,
\end{enumerate}
the exact asymptotics of Green functions was obtained, and a similar result as in \cite{IL10} on the Martin boundary was derived. In \cite{IL10,KR11}, no explicit expressions for the harmonic functions were provided.
 
For random walks with zero drift, the results are rarer, and typically require a strong underlying structure: the random walks are cartesian products in \cite{PW}, they are associated with Lie algebras in \cite{Bi1,Bi3}, etc. Last but not least, knowing the harmonic functions for zero drift random walks in ${\bf N}^{d-1}$ is necessary for constructing harmonic functions of walks with drift in ${\bf N}^d$, see \cite{IR2}. The first systematical result was obtained in \cite{Ra14}: under~\ref{prop4}--\ref{prop6}, where~\ref{prop6} is the zero drift hypothesis
\begin{enumerate}[label={\rm (\alph{*})},ref={\rm (\alph{*})}]
\setcounter{enumi}{2}
\item\label{prop6}$\sum_{k,\ell}kp_{k,\ell}=\sum_{k,\ell}\ell p_{k,\ell}=0$,
\end{enumerate}
it was proved that there is a unique discrete harmonic function (up to multiplicative factors). In \cite{Ra14}, there is also an explicit expression for the generating function
\begin{equation}
\label{eq:generating_functions_harmonic_functions}
     H(x,y)=  \sum_{i,j\geq 1} f(i,j) x^{i-1}y^{j-1}
\end{equation}
of the values of the harmonic function. Finally, this uniqueness result is extended in \cite{BMS15} to a much larger class of transition probabilities and dimension.

The are less examples of studies of $t$-Martin boundary. One of them is \cite{IR10}, for reflected random walks in half-spaces.

\subsection*{Main results}
Our main results are on the structure of the $t$-Martin boundary (Theorem~\ref{thm:main_1}) and on the explicit expression of the $t$-harmonic functions (Theorem~\ref{thm:main_2}). 

Define $t_0$ by
\begin{equation}
\label{eq:def_t0}
     t_0 = \min_{a\in{\bf R}^2}\phi(a),
\end{equation}
where we have noted
\begin{equation}
\label{eq:def_phi}
     \phi (a) = \phi(a_1,a_2) = \sum_{k,\ell} p_{k,\ell} e^{k a_1}e^{\ell a_2}.
\end{equation}
Notice that $t_0\in (0,1]$, and $t_0=1$ if and only if~\ref{prop6} holds (Figure~\ref{fig:loc_t_0}).

\unitlength=0.6cm
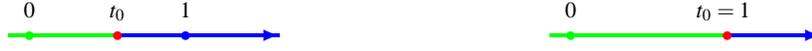
\begin{figure}[t]
\begin{center}
    \begin{picture}(0,1)   
    \thicklines
    \put(-9,0){\textcolor{green}{\line(1,0){2.35}}}
  \put(-6.65,0){\textcolor{blue}{\vector(1,0){3.65}}}
   \put(3,0){\textcolor{green}{\line(1,0){3.85}}}
        \put(6.85,0){\textcolor{blue}{\vector(1,0){2.15}}}
    \put(-8.66,0.4){$0$}
     \put(-6.75,0.4){$t_0$}
      \put(-5.2,0.4){$1$}
     {\put(-8.66,-0.13){\textcolor{green}{$\bullet$}}}
     {\put(-6.7,-0.13){\textcolor{red}{$\bullet$}}}
    {\put(-5.2,-0.13){\textcolor{blue}{$\bullet$}}}
    \put(3.34,0.4){$0$}
      \put(6.25,0.4){$t_0=1$}
     {\put(3.34,-0.13){\textcolor{green}{$\bullet$}}}
    {\put(6.8,-0.13){\textcolor{red}{$\bullet$}}}
     \end{picture}
\end{center}
\caption{Location of $t_0$ in the non-zero drift case (left) and in the zero drift case (right). There are three regimes for the $t$-Martin boundary: empty (green, left), reduced to one point (red, middle), homeomorphic to a segment (blue, right), see Theorem~\ref{thm:main_1}}
\label{fig:loc_t_0}
\end{figure}

\begin{theorem}
\label{thm:main_1}
For any random walk satisfying to~\ref{prop4}--\ref{prop5}, the $t$-Martin boundary is,
\begin{enumerate}
     \item\label{thm:main_1_1}for $t>t_0$, homeomorphic to a segment $S_t$ (with non-empty interior); 
     \item\label{thm:main_1_2}for $t=t_0$, reduced to one point;
     \item\label{thm:main_1_3}for $t<t_0$, empty.
\end{enumerate}     
\end{theorem}

For $t=1$ and non-zero drift, Theorem~\ref{thm:main_1}~\ref{thm:main_1_1} is proved in \cite{IL10,KR11}; for $t=1$ and zero drift, Theorem~\ref{thm:main_1}~\ref{thm:main_1_2} is obtained in \cite[Theorem 12]{Ra14}. Theorem~\ref{thm:main_1}~\ref{thm:main_1_3} follows from general results on Markov kernels, see, e.g., \cite{Pr64}.

We shall give two proofs of Theorem~\ref{thm:main_1}. The first one is based on a functional equation satisfied by the generating function~\eqref{eq:generating_functions_harmonic_functions} of any $t$-harmonic function (see~\eqref{eq:functional_equation}). This method (solving the functional equation via complex analysis) was introduced in \cite{Ra14} for the case $t=1$ and zero drift. As we shall see, it has the following advantages: it works for any $t$; the critical value $t_0$ appears very naturally; finally, it provides an expression for the $t$-harmonic functions (see our Theorem~\ref{thm:main_2}). 

The second proof is based on an exponential change of measure, which allows to reduce the general case to the case $t=1$. In particular, with this second method, Theorem~\ref{thm:main_1} can be extended to a much larger class than those satisfying to~\ref{prop4}--\ref{prop5}: namely, Theorem~\ref{thm:main_1}~\ref{thm:main_1_1} to the class of random walks whose increments have exponential moments (thanks to \cite[Corollary 1.1]{IL10}), and Theorem~\ref{thm:main_1}~\ref{thm:main_1_2} to random walks with bounded symmetric jumps (thanks to \cite[Theorem 1]{BMS15}).

Our second theorem provides an explicit expression for the generating function~\eqref{eq:generating_functions_harmonic_functions}. We recall that a $t$-harmonic function $f > 0$ is said to be minimal if for any $t$-harmonic function $g > 0$, the inequality $g\leq f$ implies the equality $g= c\cdot f$, for some $c > 0$. Notice that
\begin{equation}
\label{eq:generating_functions_harmonic_functions_uni}
     H(x,0)= \sum_{i\geq 1} f(i,1) x^{i-1},
     \qquad H(0,y)= \sum_{j\geq 1} f(1,j) y^{j-1}
\end{equation}
are the generating functions of the values of the harmonic function above/on the right of the coordinate axes. Introduce the second order (in $x$ and $y$) polynomial, called the kernel,
\begin{equation}
\label{eq:def_L}
     L(x,y)= x y\left( \sum_{-1\leq k,\ell \leq 1}p_{k,\ell }x^{-k} y^{-\ell}  -t\right).
\end{equation}
The kernel is fully characterized by the jumps $\{p_{k,\ell}\}$.

\begin{theorem}
\label{thm:main_2}
Let $\{p_{k,\ell}\}$ be any jumps satisfying to~\ref{prop4}--\ref{prop5}, and let $S_t$ be the segment in Theorem~\ref{thm:main_1}.

In case~\ref{thm:main_1_1} {\rm(}$t\in(t_0,\infty)${\rm)}, there exists a universal function $w$ {\rm(}see~\eqref{eq:expression_w} and~\eqref{eq:expression_u}{\rm)}, i.e., a function depending only on the kernel $L$ {\rm(}and therefore also on $t${\rm)}, such that for any minimal $t$-harmonic function $\{f(i,j)\}$, there exist $p\in S_t$  and two constants $\alpha,\beta$ {\rm(}see~\eqref{eq:expression_alpha} and~\eqref{eq:expression_beta}{\rm)}, with
\begin{equation*}
     H(x,0)= \frac{1}{L(x,0)}\left(\frac{\alpha}{w(x)-w(p)}+\beta\right).
\end{equation*}

In case~\ref{thm:main_1_2} {\rm(}$t=t_0${\rm)}, the $t$-harmonic function is unique, up to multiplicative factors. Its expression can be obtained either as the limit of the above expression when $t\to t_0$, or directly with Equations~\eqref{eq:expression_w} and~\eqref{eq:expression_u_0}.
\end{theorem}
A similar expression holds for $H(0,y)$, and finally the functional equation~\eqref{eq:functional_equation} gives the announced expression for $H(x,y)$. Theorem~\ref{thm:main_2} will be stated in full details in Section~\ref{sec:resolution}.

\subsection*{Organization of the paper}
In Section~\ref{sec:BVP} we state the functional equation~\eqref{eq:functional_equation}, we introduce some notation, we compute the growth of $t$-harmonic functions (Lemmas~\ref{lem:exp_growth_t=t} and~\ref{lem:exp_growth_t=t_bis}), and we finally show that the generating function~\eqref{eq:generating_functions_harmonic_functions_uni} satisfies a simple boundary value problem (Lemma~\ref{lem:boundary_condition}). In Section~\ref{sec:resolution} we solve this boundary value problem, by introducing the notion of conformal gluing functions (Definition~\ref{def:conformal_gluing}). In Section~\ref{sec:general_case} we extend our Theorem~\ref{thm:main_1} to a larger class of jumps $\{p_{k,\ell}\}$, by making an exponential change of jumps (Corollary~\ref{cor:extended}). In Section~\ref{sec:misc} we propose some remarks and a conjecture around our results. In Appendix~\ref{sec:expression_w} we give an explicit expression for the conformal gluing function $w$ of Theorem~\ref{thm:main_2}.

Our paper is self-contained. However, for some technical aspects of our work, specially those concerning random walks in the quarter plane, we decided to state the results without proof, referring the readers to the large existing literature (see, e.g., \cite{FIM,FR11,KM98,KR11,Ra14}). 

\section{Boundary value problem for the generating functions of harmonic functions}
\label{sec:BVP}

Our approach extends the one in \cite{Ra14}, and consists in using the generating function $H(x,y)$ (see~\eqref{eq:generating_functions_harmonic_functions}) of the harmonic function. The key point is that this function $H(x,y)$ satisfies the functional equation
\begin{equation}
\label{eq:functional_equation}
     L(x,y)H(x,y) = L(x,0) H(x,0) + L(0,y) H(0,y)-L(0,0) H(0,0),
\end{equation}
where $L$ is defined in~\eqref{eq:def_L}.

The proof of~\eqref{eq:functional_equation} simply comes from multiplying the relation $L_{t}(f)(i,j)=0$ by $x^{i}y^{j}$ and then from summing w.r.t.\ $i,j\geq 1$. In~\eqref{eq:functional_equation}, the variables $x$ and $y$ can be seen as formal variables, but they will mostly be used as complex variables.

This section is organized as follows: we first study important properties of the kernel~\eqref{eq:def_L}. Then we are interested in the regularity (as complex functions) of $H(x,0)$ and $H(0,y)$, which is related to the exponential growth of harmonic functions. Then we state a boundary value problem satisfied by these generating functions.

\subsection{Notations}
The kernel $L(x,y)$ in~\eqref{eq:def_L} can also be written
     \begin{equation}
     \label{eq:alternative_definition_L}
          L(x,y) = \alpha (x) y^{2}+ \beta (x) y + \gamma (x) = \widetilde{\alpha }(y) x^{2}+
          \widetilde{\beta }(y) x + \widetilde{\gamma }(y),
     \end{equation}
where (without loss of generality, we assume that $p_{0,0}=0$)
          \begin{equation*}
           \label{def_a_b_c}
          \left\{\begin{array}{rll}
\alpha (x)&=&  p_{-1,-1}x^{2}+ p_{0,-1}x+ p_{1,-1},\\
                \beta(x)  &=& p_{-1,0}x^{2}- tx+ p_{1,0},\\
               \gamma(x)  &=& p_{-1,1}x^{2}+ p_{0,1}x+ p_{1,1},\\
               \widetilde{\alpha }(y)  &=&  p_{-1,-1}y^{2}+ p_{-1,0}y+ p_{-1,1},\\
               \widetilde{\beta}(y)  &=& p_{0,-1}y^2- ty+ p_{0,1},\\
               \widetilde{\gamma}(y)  &=& p_{1,-1}y^{2}+ p_{1,0}y+ p_{1,1}.\end{array}\right.
     \end{equation*}
We also define
     \begin{equation}
     \label{def_d}
          \delta (x)=\beta (x)^{2}-4\alpha (x) \gamma(x),
          \qquad \widetilde{\delta }(y)=
          \widetilde{\beta }(y)^{2}-4
          \widetilde{\alpha }(y)\widetilde{\gamma }(y),
     \end{equation}
which are the discriminants of the polynomial $L(x,y)$ as a function of $y$ and $x$, respectively. 
The following facts regarding the polynomial $\delta$ are proved in \cite[Chapter 2]{FIM} for $t=1$, their proof for general values of $t\geq t_0$ ($t_0$ being defined in~\eqref{eq:def_t0}) would be similar: under~\ref{prop4}--\ref{prop5}, $\delta $ has degree (in $x$) three or four. We denote its roots by  $\{x_\ell\}_{1\leq \ell\leq 4}$, with
\begin{equation}
\label{eq:properties_branch_points_1}
          \vert x_1\vert\leq \vert x_2\vert \leq \vert x_3\vert\leq \vert x_4\vert,
     \end{equation}
and $x_4=\infty$ if $\delta $ has degree three. We have $x_1\in[-1,1)$, $x_4\in(1,\infty) \cup \{\infty\}\cup (-\infty,-1]$, and $x_2,x_3>0$. Further $\delta(x)$ is negative on ${\bf R}$ if and only if $x\in (x_1,x_2)\cup (x_3,x_4)$. The polynomial $\widetilde \delta$ in~\eqref{def_d} and its roots $\{y_\ell\}_{1\leq \ell\leq 4}$ satisfy similar properties.

In what follows, we define the algebraic functions $X(y)$ and $Y(x)$ by $L(X(y),y)=0$ and $L(x,Y(x))=0$. With~\eqref{eq:alternative_definition_L} and~\eqref{def_d} we have the obvious expressions
\begin{equation}
\label{eq:expression_X_Y}
          X(y)=\frac{-\widetilde \beta(y)\pm {\sqrt{\widetilde\delta(y)}}}{2 \widetilde \alpha(y)},
          \qquad Y(x)=\frac{-\beta(x)\pm \sqrt{\delta(x)}}{2 \alpha(x)}.
\end{equation}
The functions $X(y)$ and $Y(x)$ both have two branches, called $X_0,X_1$ and $Y_0,Y_1$, which are meromorphic on the cut planes ${\bf C}\setminus ([y_1,y_2]\cup[y_3,y_4])$ and ${\bf C}\setminus([x_1,x_2]\cup[x_3,x_4])$, respectively. The numbering of the branches can be chosen so as to satisfy $\vert X_0(y)\vert \leq \vert X_1(y)\vert$ (resp.\ $\vert Y_0(x)\vert \leq \vert Y_1(x)\vert$) on the whole of the cut planes, see \cite[Theorem 5.3.3]{FIM}.

Note that except $\alpha,\gamma,\widetilde\alpha,\widetilde\gamma$, all quantities defined above depend on $t$.

\subsection{Growth of $t$-harmonic functions}
By definition, the exponential growth of a sequence $\{u_i\}$ of positive real numbers is $\limsup_{i\to\infty} u_i^{1/i}$. We first identify (Lemma~\ref{lem:exp_growth_t=1}) the exponential growth of $\{f(i,1)\}$ and $\{f(1,j)\}$ for $t=1$, and then (Lemma~\ref{lem:exp_growth_t=t}) we treat the general case in $t$.  

\subsubsection*{First case: $t=1$}

Consider on ${\bf R}^2$ the function $\phi$ defined by~\eqref{eq:def_phi}, and define the set $D_1=\{a\in {\bf R}^2 : \phi(a)\leq 1\}$ and its boundary $\partial D_1=\{a\in {\bf R}^2 : \phi(a)= 1\}$.  If the drift is zero (hypothesis~\ref{prop6}), the set $D_1$ is reduced to $\{0\}$, see \cite[Proposition 4.3]{He63}. If not, it is homeomorphic to the unit disc. More precisely, for $a\in \partial D_1$, let $q(a)=\frac{\nabla \phi(a)}{\vert \nabla \phi(a)\vert}\in{\bf S}^1$ (the unit circle). If the drift is non-zero, the function $q$ is a homeomorphism between $\partial D_1$ and ${\bf S}^1$, see \cite[Proposition 4.4]{He63} or \cite[Introduction]{IL10}. Define finally ${\bf S}^1_+={\bf S}^1\cap {\bf R}_+^2$ as well as  $\Gamma^+_1=\{a\in \partial D_1 : q(a)\in {\bf S}^1_+\}$. The following result is proved in \cite{IL10}.

\begin{lemma}[\cite{IL10}]
\label{lem:exp_growth_t=1}
For any non-zero minimal $1$-harmonic function $f$, there exists $a=(a_1,a_2)\in \Gamma^+_1$ such that the exponential growth of $\{f(i,1)\}$ {\rm(}resp.\ $\{f(1,j)\}${\rm)} is $a_1$ {\rm(}resp.\ $a_2${\rm)}.
\end{lemma}

(And reciprocally, any $a\in \Gamma^+_1$ is the growth of a minimal $1$-harmonic function.) Lemma~\ref{lem:exp_growth_t=1} follows from Equation (1.3) in \cite{IL10}, which gives the structure of any minimal harmonic function. It holds a priori only in the case of a non-zero drift, but it turns out to be also true in the zero drift case, as there is then no exponential growth, i.e., $a_1=a_2=0$, which is guaranteed by \cite[Lemma 2]{Ra14}.

\subsubsection*{General case in $t$} We introduce $D_t=\{a\in {\bf R}^2 : \phi(a)\leq t\}$ as well as (with obvious notation) $\partial D_t$ and $\Gamma^+_t$.

The function $\phi$ is strictly convex on ${\bf R}^2$, and due to the hypothesis~\ref{prop5} it admits a global minimum on ${\bf R}^2$. Let $t_0$ as in~\eqref{eq:def_t0}.
Note that $t_0\leq 1$ (evaluate $\phi$ at $0$) and that $t_0=1$ if and only if the drift is zero (see \cite[Proposition 4.3]{He63}). The following result extends Lemma~\ref{lem:exp_growth_t=1} to $t$-harmonic functions.
\begin{lemma}
\label{lem:exp_growth_t=t}
Let $t\geq t_0$. For any non-zero minimal $t$-harmonic function $f$, there exists $a\in \Gamma^+_t$ such that the exponential growth of $\{f(i,1)\}$ {\rm(}resp.\ $\{f(1,j)\}${\rm)} is $a_1$ {\rm(}resp.\ $a_2${\rm)}.
\end{lemma}

(And reciprocally, any $a\in \Gamma^+_t$ is the growth of a minimal $t$-harmonic function.) Lemma~\ref{lem:exp_growth_t=t} could be proved along the same lines as Lemma~\ref{lem:exp_growth_t=1}, but it can also be obtained thanks to the exponential change of the parameters $\{p_{i,j}\}$ presented in Section~\ref{sec:general_case}. As Lemma~\ref{lem:exp_growth_t=1}, Lemma~\ref{lem:exp_growth_t=t} holds for any value of the drift.

\subsubsection*{Reformulation in terms of the kernel}

Lemma~\ref{lem:exp_growth_t=t} can be reformulated as follows, in terms of quantities related to the kernel~\eqref{eq:def_L}. This will be more convenient for our analysis.
\begin{lemma}
\label{lem:exp_growth_t=t_bis}
Let $t\geq t_0$. For any non-zero minimal $t$-harmonic function $f$, there exists $p\in[x_2,X(y_2)]$ {\rm(}resp.\ $p'\in[y_2,Y(x_2)]${\rm)} with $p'=Y_0(p)$ {\rm(}or $p=X_0(p')${\rm)}, such that the exponential growth of $\{f(i,1)\}$ {\rm(}resp.\ $\{f(1,j)\}${\rm)} is $1/p$ {\rm(}resp.\ $1/p'${\rm)}. 
\end{lemma}

\begin{proof}
We first notice that $\phi(a)=t$ if and only if $L(1/e^{a_1},1/e^{a_2})=0$, see~\eqref{eq:def_phi} and~\eqref{eq:def_L}. Moreover, the real and positive points of $\{(x,y)\in{\bf C}^2 : L(x,y)=0\}$ are (see \cite{KM98,FIM})
\begin{equation*}
     \mathcal P= \{(x,Y_0(x)) : x\in[x_2,x_3]\}\cup \{(x,Y_1(x)) : x\in[x_2,x_3]\}.
\end{equation*}
The fact that $a\in \Gamma^+_t$ implies that $x\in[x_2,X(y_2)]$, since the normal to the curve $\mathcal P$ at $x_2$ (resp.\ $X(y_2)$) is $(-1,0)$ (resp.\ $(0,-1)$).\qed
\end{proof}

As a consequence of Lemmas~\ref{lem:exp_growth_t=t} and~\ref{lem:exp_growth_t=t_bis}, we obtain a proof of Theorem~\ref{thm:main_1}~\ref{thm:main_1_1}. We shall give more details on the proof of Theorem~\ref{thm:main_1} in Section~\ref{sec:resolution}.

\subsection{A boundary value problem}

In this section we prove that the function $L(x,0)H(x,0)$ satisfies a simple boundary value problem. A boundary value problem is composed of a boundary condition (Lemma~\ref{lem:boundary_condition}) and a regularity condition (Lemma~\ref{lem:regularity_gf}).

With the previous notation we introduce 
\begin{equation*}
     \mathcal M=X([y_1,y_2])=X_0([y_1,y_2])\cup X_1([y_1,y_2]).
\end{equation*}
This curve is symmetrical w.r.t.\ the real axis, since $\widetilde \delta$ is non-positive on $[y_1,y_2]$, and hence the two branches $X_0$ and $X_1$ are complex conjugate on that interval. See Figure~\ref{fig:ex_curve} for an example of curve $\mathcal M$. 

Denote by $\overline{x}$ the complex conjugate of $x\in{\bf C}$.
\begin{lemma}
\label{lem:boundary_condition}
We have the boundary condition: for all $x$ in $\mathcal M$,
\begin{equation*}
     L(x,0)H(x,0)-L(\overline{x},0)H(\overline{x},0) = 0.
\end{equation*}
\end{lemma}
We have a similar equation for $L(0,y)H(0,y)$ on the curve $\mathcal L=Y([x_1,x_2])$. 
\begin{proof}
Lemma~\ref{lem:boundary_condition} is classical; see \cite[Section 2.6]{Ra14} for the original proof in the zero drift case. The main idea is to evaluate~\eqref{eq:functional_equation} at $(X_0(y),y)$, and then to make the difference of the two equations obtained by letting $y$ go to $[y_1,y_2]$ from above and below in ${\bf C}$ (i.e., with $y$ having a positive and then a negative imaginary part).\qed
\end{proof}

\unitlength=0.6cm
\begin{figure}[t]
\begin{center}
    \begin{picture}(0,5)
    \thicklines
    \put(-5,0){\vector(1,0){10}}
    \put(0,-5){\vector(0,1){10}}
     \put(3.5,0.4){$X(y_2)$}
     \put(-3.3,0.4){$X(y_1)$}
    {\put(0.5,0){\textcolor{black}{\circle*{0.3}}}}
    \put(1.3,0.4){$x_2$}
    \put(2.2,-0.65){\textcolor{red}{$S_t$}}
    \put(-1.0,2.8){\textcolor{blue}{$\mathcal M$}}
   \textcolor{blue}{\qbezier(3.5,0)(2,3.5)(0,3.5)
    \qbezier(3.5,0)(2,-3.5)(0,-3.5)
    \qbezier(0,3.5)(-3.5,3.5)(-3.5,0)
    \qbezier(-3.5,0)(-3.5,-3.5)(0,-3.5)}
            \linethickness{1mm}
    \textcolor{red}{\put(1.1,0){\line(1,0){2.15}}}
    \thicklines
     {\put(1,-0.13){\textcolor{black}{$\bullet$}}}
    {\put(3.10,-0.13){\textcolor{black}{$\bullet$}}}
    {\put(-3.87,-0.13){\textcolor{black}{$\bullet$}}}
    \thicklines
     \end{picture}
\end{center}
\vspace{28mm}
\caption{The curve $\mathcal M=X([y_1,y_2])$ is symmetrical w.r.t.\ the real axis. It is smooth everywhere except at $X(y_2)$, where it may have a corner point (if and only if $t=t_0$). Any $1/p$, with $p$ point of the red segment $S_t$, is the exponential growth of a harmonic function}
\label{fig:ex_curve}
\end{figure}
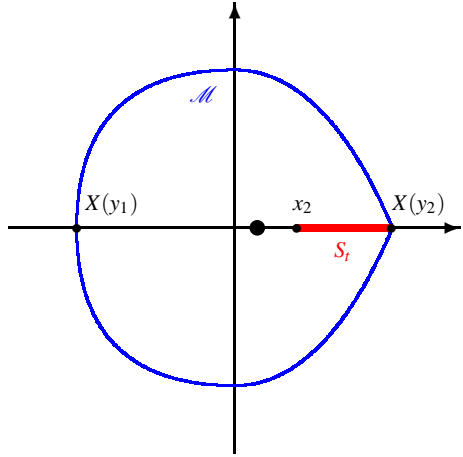

Lemma~\ref{lem:exp_growth_t=t_bis} implies that: to any minimal $t$-harmonic function $\{f(i,j)\}$ we can associate a number $p\in [x_2,X(y_2)]$ such that $1/p$ (resp.\ $1/p'$, with $p'=Y_0(p)$) is the exponential growth of $\{f(i,1)\}$ (resp.\ $\{f(1,j)\}$). We write $\{f_p(i,j)\}$ and $H_p(x,y)$ to emphasize this exponential growth. 

Let $\mathcal D_\mathcal M$ be the interior domain delimited by the curve $\mathcal M$ (containing $x_2$ on Figure~\ref{fig:ex_curve}) and $\overline{\mathcal D}_\mathcal M=\mathcal D_\mathcal M \cup \mathcal M$ be its closure. The lemma hereafter gives the regularity of the complex function $H_p(x,0)$ (a similar result holds for $H_p(0,y)$).
\begin{lemma}
\label{lem:regularity_gf}
For any $t\in[t_0,\infty)$, the generating function $H_p(x,0)$ is meromorphic in $\mathcal D_\mathcal M$, and has in $\overline{\mathcal D}_\mathcal M$ a unique singularity, at $p$. The singularity is on the boundary $\mathcal M$ if and only if $t\neq t_0$ and $p=X(y_2)$ or if $t=t_0$. In case~\ref{thm:main_1_1} {\rm(}$t\neq t_0${\rm)}, the singularity is polar. In case~\ref{thm:main_1_2} {\rm(}$t=t_0${\rm)}, it can be polar or not polar.
\end{lemma}

\begin{proof}
The case $t=t_0$ is rather special. In the case of a zero drift ($t_0=1$), it has been proved in \cite[Lemma 3]{Ra14}. For other values of $t=t_0$, the proof would be completely similar.

We therefore assume that $t\neq t_0$. It follows from Lemma~\ref{lem:exp_growth_t=t_bis} that the function $H_p(x,0)$ is analytic in the open disc $\mathcal D(0,p)$ centered at $0$ and of radius $p$. The same holds for $H_p(0,y)$ in $\mathcal D(0,p')$. Consider the identity
\begin{equation}
\label{eq:func_eq_evaluated}
     L(x,0)H_p(x,0) + L(0,Y_0(x))H_p(0,Y_0(x))-L(0,0)H_p(0,0)=0,
\end{equation}
which is the functional equation~\eqref{eq:functional_equation} evaluated at $(x,Y_0(x))$. The fact that~\eqref{eq:func_eq_evaluated} holds on a non-empty set is not clear a priori, and follows from Lemma~\ref{lem:curve_included}, with $x$ on the circle of radius $p$. A consequence of~\eqref{eq:func_eq_evaluated} is that $H_p(x,0)$ can be continued on the whole of $\mathcal D_\mathcal M$. Indeed, writing
\begin{equation*}
     \mathcal D_\mathcal M = \mathcal D(0,p) \cup (\mathcal D_\mathcal M\setminus \mathcal D(0,p)),
\end{equation*}
the generating function is defined through its power series in the first domain, and thanks to $H_p(0,Y_0(x))$ in the complementary domain. Further, we have that $\lim_{x\to p}\vert H_p(x,0)\vert =\infty$ (independently of the way that $x\to p$), so that $p$ is indeed a pole, and not an essential singularity.\qed
\end{proof}

The following result has been used in the proof of Lemma~\ref{lem:regularity_gf}. For the proof we refer to \cite[Lemma 28]{KR11}, which is a very close statement.

\begin{lemma}
\label{lem:curve_included}
Let $x\in [x_2,X(y_2)]$. Then for all $\vert u\vert=x$, we have $\vert Y_0(u)\vert\leq Y_0(\vert u\vert )$. Furthermore, the inequality is strict, except for $u=x$.
\end{lemma}

\section{Resolution of the boundary value problem: proof of Theorems~\ref{thm:main_1} and~\ref{thm:main_2}}
\label{sec:resolution}

\subsection{Conformal gluing functions}

Lemmas~\ref{lem:boundary_condition} and~\ref{lem:regularity_gf} imply that the function $L(x,0)H(x,0)$ belongs to the set of functions $f$ which are meromorphic in $\mathcal D_\mathcal M$ and satisfy on $\mathcal M$ the equality $f(x)=f(\overline{x})$. This set of functions is too large to work on: for instance, $P\circ f$ still belongs to this set for any polynomial $P$. The good idea is to impose a minimality condition on $f$, and to introduce the notion of conformal gluing functions (our general reference for this is the book of Litvinchuk \cite{Li00}, and more specifically its second chapter).

\begin{definition}
\label{def:conformal_gluing}
A conformal gluing function $w$ for $\mathcal D_\mathcal M$ is a function meromorphic and injective on $\mathcal D_\mathcal M$, continuous on $\overline{\mathcal D}_\mathcal M$ except at a finite number of points, and such that $w(x)=w(\overline x)$ for $x\in\mathcal M$.
\end{definition}

As stated in the lemma below, conformal gluing functions exist. They must have a unique singularity on $\overline{\mathcal D}_\mathcal M$ (of order $1$ if the singularity is in the interior $\mathcal D_\mathcal M$), and are essentially characterized by the location of this singularity.

\begin{lemma}[\cite{Li00}]
\label{lem:conformal_mapping_exists}
Let $p\in\overline{\mathcal D}_\mathcal M$. Up to additive and multiplicative constants, there exists a unique conformal gluing function $w$ for $\mathcal D_\mathcal M$ with a pole at $p$. Further, for any two conformal gluing functions $w_1$ and $w_2$, there exist $a,b,c,d\in{\bf C}$ with $ad-bc\neq 0$ such that $w_2=\frac{aw_1+b}{cw_1+d}$.
\end{lemma}

\subsection{Complete statement of Theorem~\ref{thm:main_2}}

Let $w$ be a conformal mapping as in Lemma~\ref{lem:conformal_mapping_exists} with a pole at $x_0\in (X(y_1),x_2)\setminus\{0\}$ (this reference point $x_0$ is arbitrary), see Figure~\ref{fig:ex_curve} for its location. Subtracting by $w(0)$, we may assume that $w(0)=0$.

The singularity of $L(x,0)H_p(x,0)$ is not located anywhere in $\overline{\mathcal D}_\mathcal M$, but on the segment $[x_2,X(y_2)]$, see Lemma~\ref{lem:regularity_gf}. Let us call $\frac{\alpha}{w-w(p)}+\beta$ the class of conformal gluing functions with a pole at $p\in[x_2,X(y_2)]$, see Lemma~\ref{lem:conformal_mapping_exists}. Our Theorem~\ref{thm:main_2} will be restated as:
\begin{equation}
\label{eq:our_result_shows}
     L(x,0)H_p(x,0)=\frac{\alpha}{w(x)-w(p)}+\beta.
\end{equation}
In other words, the conformal gluing functions parametrized by $p\in[x_2,X(y_2)]$ offer a complete solution to our problem. Notice that expressions for the constants $\alpha$ and $\beta$ will follow from a one or two term(s) expansion of the equality~\eqref{eq:our_result_shows}.

We now state the theorem in full details. Define
\begin{equation}
\label{eq:expression_alpha}
     \alpha=-f(1,1)\times \left\{\begin{array}{ll}
        \displaystyle\frac{p_{0,1}w(p)^2}{w'(0)} &\text{if } p_{1,1}=0 \text{ and } p_{0,1}\neq 0,\\
        \displaystyle\frac{2p_{-1,1}w(p)^2}{w''(0)}&\text{if } p_{1,1}=0 \text{ and } p_{0,1}= 0,\smallskip\\
        \displaystyle\frac{(w(X_0(0))-w(p))w(p)}{w(X_0(0))}&\text{if } p_{1,1}\neq 0,
        \end{array}\right.
\end{equation}
and
\begin{equation}
\label{eq:expression_beta}
\beta=p_{1,1}f(1,1)+\frac{\alpha}{w(p)}.
\end{equation}

\setcounter{theorem}{1}
\begin{theorem}[complete version]
\label{thm:main_2_complete}
Let $\alpha$ and $\beta$ be defined in~\eqref{eq:expression_alpha} and~\eqref{eq:expression_beta}. We have
\begin{align*}
     H_p(x,0) &= \frac{1}{L(x,0)}\left(\frac{\alpha}{w(x)-w(p)}+\beta\right),\\ 
     H_p(0,y)&=\frac{p_{1,1}f(1,1)-L(X_0(y),0)H_p(X_0(y),0)}{L(0,y)}.
\end{align*}
\end{theorem}

\subsection{Proof of Theorem~\ref{thm:main_2_complete}}

\begin{proof}
The proof of Theorem~\ref{thm:main_2_complete} is based on the following remark (see \cite[Lemma 4]{Ra14}, taken from \cite{Li00}): if a function $f$ is analytic on $\mathcal D_\mathcal M$, continuous on $\overline{\mathcal D}_\mathcal M$ and satisfies $f(x)=f(\overline{x})$ for $x\in\mathcal M$, then it must be a constant function.

Let us begin with the case $t>t_0$. Lemma~\ref{lem:regularity_gf} implies that $L(x,0)H_p(x,0)$ has a unique pole at $p$. Let us assume for a while that this pole is of order $1$ (which will always be the case, except if $p=X(y_2)$, in which case the pole has order $2$). Then by choosing suitably the value of $\alpha$, the function
\begin{equation}
\label{eq:after_compensation}
     L(x,0)H_p(x,0)-\frac{\alpha}{w(x)-w(p)}-\beta
\end{equation}
has no pole in $\mathcal D_\mathcal M$ (since $w$ is injective in $\mathcal D_\mathcal M$, see Definition~\ref{def:conformal_gluing}, the function $\frac{1}{w(x)-w(p)}$ has a pole of order $1$ at $p$, as soon as $p\in\mathcal D_\mathcal M$) and is continuous on $\overline{\mathcal D}_\mathcal M$. The function~\eqref{eq:after_compensation} also satisfies the condition $f(x)=f(\overline{x})$ on the boundary. Hence we can use the above remark to conclude that~\eqref{eq:after_compensation} is a constant function. The value of $\beta$ can be adapted so as to have that~\eqref{eq:after_compensation} is $0$. To compute the exact values of the constants $\alpha$ and $\beta$, a series expansion of~\eqref{eq:our_result_shows} around $0$ is enough.

In the case $t>t_0$ but $p=X(y_2)$, the pole of $\frac{1}{w(x)-w(p)}$ at $p$ is of order $2$ (indeed, around $p\in\{X(y_1),X(y_2)\}$, the equality $w(x)=w(\overline{x})$ yields $w'(p)=0$), and the same expression as~\eqref{eq:after_compensation} can be obtained.

The fact that $L(x,0)H_p(x,0)$ has a pole of order $1$ or $2$ at $p$ follows essentially from that we are looking for positive harmonic functions $\{f(i,j)\}$. If the pole were of higher order, then the solution would have negative coefficients in its expansion near $0$, see \cite[Lemma 11]{Ra14}, which is impossible.

See \cite[Sections 3.1--3.3]{Ra14} for the proof of Theorem~\ref{thm:main_2_complete} in the case $t=t_0=1$, which can be immediately adapted to the case $t=t_0\neq 1$.\qed
\end{proof}

\subsection{Proof of Theorem~\ref{thm:main_1}}

\begin{proof}
In the small steps case (assumptions~\ref{prop4}--\ref{prop5}),~\ref{thm:main_1_1} and~\ref{thm:main_1_2} of Theorem~\ref{thm:main_1} follow independently from Theorem~\ref{thm:main_2} or from Lemma~\ref{lem:exp_growth_t=t} (or its reformulation Lemma~\ref{lem:exp_growth_t=t_bis}). Theorem~\ref{thm:main_1}~\ref{thm:main_1_3} is a consequence of classical results, see \cite{Pr64}.\qed
\end{proof}

\section{Extension and second proof of Theorem~\ref{thm:main_1}}
\label{sec:general_case}

In this section we consider weights $\{p_{k,\ell}\}$ having exponential moments. This assumption implies that the function $\phi$ introduced in~\eqref{eq:def_phi} is well defined on ${\bf R}^2$.

We note $\{f[p_{k,\ell}]\}$ the set of $1$-harmonic functions, once the jumps $\{p_{k,\ell}\}$ have been fixed. For any $a$ such that $\phi(a)=t$, we define new weights as follows (with $\langle \cdot,\cdot\rangle$ denoting the standard scalar product in ${\bf R}^2$)
\begin{equation}
\label{eq:pkla}
     p_{k,\ell}^{a} = p_{k,\ell} e^{\langle a,(k,\ell)\rangle} t^{-1}.
\end{equation}
The identity $\phi(a)=t$ implies that $\sum_{k,\ell} p_{k,\ell}^{a} =1$, and thus the $\{p_{k,\ell}^{a}\}$ can be interpreted as transition probabilities. Our main result in Section~\ref{sec:general_case} is the following:
\begin{proposition}
\label{prop:t-harmonic_second_approach}
Assume that the $\{p_{k,\ell}\}$ have all exponential moments. Then the set of $t$-harmonic functions is equal to 
\begin{equation*}
     \{(i,j)\mapsto e^{\langle a_t,(i,j)\rangle} f[p_{k,\ell}^{a_t}](i,j)\},
\end{equation*}
for any $a_t$ such that $\phi(a_t)=t$.
\end{proposition}
Proposition~\ref{prop:t-harmonic_second_approach} is a direct consequence of the following simple correspondence between $t$-harmonic and $1$-harmonic functions. Hereafter, we shall denote by $f^{a}$ ($a\in{\bf R}^2$) the function 
\begin{equation}
\label{eq:new_harmonic}
     f^{a}(i,j) = f(i,j)e^{-\langle a,(i,j)\rangle}.
\end{equation}

\begin{lemma}
\label{lem:change_harmonic}
For any $a_t$ such that $\phi(a_t)=t$ and any $t$-harmonic function $f$, $f^{a_t}$ is $1$-harmonic w.r.t.\ the weights $\{p_{k,\ell}^{a_t}\}$. 
\end{lemma}

As a consequence of Proposition~\ref{prop:t-harmonic_second_approach} and Lemma~\ref{lem:change_harmonic}, we can reprove and extend Theorem~\ref{thm:main_1}.

\begin{corollary}
\label{cor:extended}
Theorem~\ref{thm:main_1}, initially proved for small steps random walks, can be generalized as follows:
\begin{itemize}
\item Theorem~\ref{thm:main_1}~\ref{thm:main_1_1} to random walks whose increments have exponential moments,
\item Theorem~\ref{thm:main_1}~\ref{thm:main_1_2} to random walks with bounded symmetric jumps. 
\end{itemize}
\end{corollary}

\begin{proof}
We first assume that the equation $\phi(a_t)=t$ has a unique solution. In this case $a_t$ is the global minimizer of $\phi$ on ${\bf R}^2$ and the new weights $\{p_{k,\ell}^{a_t}\}$ have zero drift (this corresponds to $\phi'(a_t)=0$). For random walks with bounded symmetric jumps, there exists a unique $f[p_{k,\ell}^{a_t}]$ which is $1$-harmonic (up to multiplicative factors), see \cite[Theorem 1]{BMS15}. Corollary~\ref{cor:extended} follows in this case.

We now suppose that the equation $\phi(a_t)=t$ has more than one solution (and then in fact, infinitely many). In this case the $\{p_{k,\ell}^{a_t}\}$ have non-zero drift (independently of $a_t$). For any choice of $a_t$, we can use the result \cite[Corollary 1.1]{IL10} for $t=1$ (valid for random walks whose increments have exponential moments), and then with~\eqref{eq:new_harmonic} and Lemma~\ref{lem:change_harmonic} we transfer it to other values of $t>t_0$.\qed
\end{proof}

We now present some remarks and consequences of Proposition~\ref{prop:t-harmonic_second_approach}:
\begin{itemize}
     \item Proposition~\ref{prop:t-harmonic_second_approach} is independent of the choice of $a_t$.
     \item Proposition~\ref{prop:t-harmonic_second_approach} is not only a result on the structure of the Martin boundary, it also provides an expression of the $t$-harmonic functions in terms of the $1$-harmonic functions.
     \item The exponential factor in~\eqref{eq:new_harmonic} does not affect the fact that on the boundary of the quadrant, the functions $f$ and $f^{a}$ are $0$. Incidentally, this explains that the simple exponential change~\eqref{eq:new_harmonic} cannot be used in other situations than killed random walks, like reflected random walks on a quadrant (see \cite{IR10} for the study of the $t$-Martin boundary of reflected random walks on a half-space).
\end{itemize}

\section{Miscellaneous} 
\label{sec:misc}

\subsection{Stable Martin boundaries}

According to \cite[Definition 2.4]{PW}, the Martin boundary is stable if the Martin compactification does not depend on the eigenvalue $t$ (with a possible exception at the critical value) and if the Martin kernels are jointly continuous w.r.t.\ space variable and eigenvalue.

The first item is clearly satisfied in our context (see our Theorem~\ref{thm:main_1}). As for the second one, it does not formally come from our results. However, it is most probably true (in this direction, see Section~\ref{sec:remarks}, where we show that the harmonic functions are continuous w.r.t.\ the eigenvalue $t$). For small steps random walks and $t=1$, it is proved in \cite[Remark 29]{KR11} that the Martin kernel is continuous w.r.t.\ the space variable.

\subsection{Transformations of the step set and consequences on harmonic functions}

It is natural to make some transformations of the step set, as $\{p_{k,\ell}\}\to \{p_{\pm k,\pm \ell}\}$, and to see the effect on the harmonic functions $\{f(i,j)\}$. In fact, the consequence will be simpler to read on the generating functions $H(x,0)$ and $H(0,y)$, without obvious implications on the coefficients $\{f(i,j)\}$.

The starting point of all our approach is the functional equation~\eqref{eq:functional_equation}, and the difference between two functional equations associated with different jumps is all contained in the kernel~\eqref{eq:def_L}.

Consider first the transformation $\{p_{k,\ell}\}\to \{p_{-k,\ell}\}$. The new kernel is $x^2L(1/x,y)$, with new branch points in $x$ equal to the $1/x_\ell$, while the $y_\ell$ remain the same. The new roots of the kernel are $1/X(y)$ and $Y(1/x)$. The curve $\mathcal L$ is the same, and the new $\mathcal M$ is obtained by an inversion. 

The new conformal mapping is an algebraic function of $w$. To find it we can proceed as in the proof of Theorem~\ref{thm:main_2_complete}, by compensating the poles of $w$ in the new curve $\mathcal M$ (see \eqref{eq:after_compensation}). Changing accordingly the values of the constants $\alpha$ and $\beta$ and replacing $L(x,0)$ by $x^2L(1/x,0)$ yields the correct statement of Theorem~\ref{thm:main_2_complete} for the step set $\{p_{-k,\ell}\}$.

Regarding the transformation $\{p_{k,\ell}\}\to \{p_{k,-\ell}\}$, $w$ takes the same value but $L(x,0)$, which is equal to $\gamma(x)$ in the case $\{p_{k,\ell}\}$, should be $\alpha(x)$.

Similar facts can be obtained for other transformations, or for the symmetry $\{p_{k,\ell}\}\to \{p_{\ell,k}\}$.

\subsection{Simple random walks}

If $p_{0,1}+p_{1,0}+p_{0,-1}+p_{-1,0}=1$, the minimal $t$-harmonic functions take the form (with $p\in[x_2,X(y_2)]$ and $p'=Y_0(p)$)
\begin{equation}
\label{eq:expression_harmonic_coefficients}
     f_p (i,j)=\left\{\begin{array}{ll}
     \big\{\big(\frac{1}{p}\big)^{i}-\big(\frac{p_{-1,0}}{p_{1,0}}p\big)^{i}\big\}j\big(\frac{1}{p'}\big)^{j}& \text{if }  p = x_2,\\
     \big\{\big(\frac{1}{p}\big)^{i}-\big(\frac{p_{-1,0}}{p_{1,0}}p\big)^{i}\big\}\big\{\big(\frac{1}{p'}\big)^{j}-\big(\frac{p_{0,-1}}{p_{0,1}}p'\big)^{j}\big\}    & \text{if }  p\in(x_2,X(y_2)),\\
     i\big(\frac{1}{p}\big)^{i}\big\{\big(\frac{1}{p'}\big)^{j}-\big(\frac{p_{0,-1}}{p_{0,1}}p'\big)^{j}\big\}& \text{if }  p=X(y_2).
     \end{array}\right.
\end{equation}
In the particular case $t=1$, Equation~\eqref{eq:expression_harmonic_coefficients} is obtained in \cite[Section 5.1]{KR11}. By using techniques coming from representation theory, the authors of \cite{LPP13} have obtained an explicit expression for one $1$-harmonic function, the one equal to the probability of never hitting the cone.

The explicit expression~\eqref{eq:expression_harmonic_coefficients} could also be obtained from Section~\ref{sec:expression_w} (via the computation of the generating functions $H_p(x,y)$), where we derive an expression for the function $w$.

\subsection{Generating functions of discrete harmonic functions as Tutte's invariants}

The equality $L(x,0)H(x,0)=L(\overline{x},0)H(\overline{x},0)$ for $x\in\mathcal M$ (Lemma~\ref{lem:boundary_condition}) implies that, in the terminology of \cite{Be}, $L(x,0)H(x,0)$ is a Tutte's invariant (these invariants were introduced in the 1970s, when studying properly $q$-colored triangulations, see \cite{Tutte}). 

In \cite{Be} the authors identity some models of quadrant walks such that their generating function can be expressed in terms of such Tutte's invariants. This illustrates that our results do not only concern Martin boundary theory, but also combinatorial problems as the enumeration of walks in the quarter plane.

\subsection{A conjecture}

Our conjecture is that Theorem~\ref{thm:main_1}, a priori valid only for a subclass of jumps $\{p_{k,\ell}\}$ with exponential moments, can be extended as follows:
\begin{conjecture}
Theorem~\ref{thm:main_1} is valid for any $\{p_{k,\ell}\}$ such that $\sum_{k,\ell}(k^2+\ell^2) p_{k,\ell}<\infty$ {\rm(}i.e., with moments of order $2${\rm)}.
\end{conjecture}
In the particular case of zero drift jumps $\{p_{k,\ell}\}$ (i.e., $t=t_0=1$), this conjecture is stated in \cite[Conjecture 1]{Ra14}.

\section*{Acknowledgments}
We would like to thank Irina Ignatiouk-Robert for very interesting discussions, in particular for suggesting Proposition~\ref{prop:t-harmonic_second_approach} and Lemma~\ref{lem:change_harmonic}. We acknowledge support from the project MADACA of the R\'egion Centre. The second author would like to thank Gerold Alsmeyer and the Institut f\"ur Mathematische Statistik (Universit\"at M\"unster, Germany), where the project has started. We finally thank a referee for useful comments.

%
%

\begin{thebibliography}{99.}%
%
%
\bibitem{Be}
Bernardi, O., Bousquet-M\'elou, M. and Raschel, K.:
Counting quadrant walks via Tutte's invariant method.
{\it Preprint arXiv:1511.04298}, 1--13 (2015)

\bibitem{Bi1}
Biane, P.:
\newblock Quantum random walk on the dual of {${\rm SU}(n)$}.
\newblock{\em Probab. Theory Related Fields} {\bf 89}, 117--129 (1991)

\bibitem{Bi3}
Biane, P.:
\newblock Minuscule weights and random walks on lattices.
\newblock In {\em Quantum probability \& related topics},
51--65. World Sci. Publ., River Edge, NJ (1992)

\bibitem{BMS15}
Bouaziz, A., Mustapha, S.  and Sifi, M.:
Discrete harmonic functions on an orthant in ${\bf Z}^d$.
{\it Electron. Commun. Probab.} {\bf20}, 1--13 (2015)

\bibitem{FIM}
Fayolle, G., Iasnogorodski, R. and Malyshev, V.:
\newblock {\em Random walks in the quarter plane}.
\newblock Springer-Verlag, Berlin (1999)

\bibitem{FR11}
Fayolle, G. and Raschel, K.:
\newblock Random walks in the quarter-plane with zero drift: an explicit criterion for the finiteness of the associated group. 
\newblock{\em Markov Process. Related Fields} {\bf17}, 619--636 (2011)

\bibitem{He63}
Hennequin, P.-L.:
Processus de Markoff en cascade. 
{\it Ann. Inst. H. Poincar\'e} {\bf18}, 109--195 (1963)

\bibitem{IR0}
Ignatiouk-Robert, I.: 
Martin boundary of a killed random walk on a half-space. 
{\it J. Theoret. Probab.} {\bf 21}, 35--68 (2008)

\bibitem{IR2}
Ignatiouk-Robert, I.: 
Martin boundary of a killed random walk on ${\bf Z}^d$.
{\it Preprint arXiv:0909.3921}, 1--49 (2009)

\bibitem{IR1}
Ignatiouk-Robert, I.: 
Martin boundary of a reflected random walk on a half-space. 
{\it Probab. Theory Related Fields} {\bf 148}, 197--245 (2010)

\bibitem{IR10}
Ignatiouk-Robert, I.:
$t$-Martin boundary of reflected random walks on a half-space. 
{\em Electron. Commun. Probab.} {\bf15}, 149--161 (2010)

\bibitem{IL10}
Ignatiouk-Robert, I. and Loree, C.:
Martin boundary of a killed random walk on a quadrant. 
{\it Ann. Probab.} {\bf38}, 1106--1142 (2010)

\bibitem{KM98}
Kurkova, I. and Malyshev, V.:
Martin boundary and elliptic curves.
{\it Markov Process. Related Fields} {\bf4}, 203--272 (1998)

\bibitem{KR11}
Kurkova, I. and Raschel, K.: 
Random walks in ${\bf Z}_+^2$ with non-zero drift absorbed at the axes. 
{\it Bull. Soc. Math. France} {\bf139}, 341--387 (2011)

\bibitem{LPP13}
Lecouvey, C., Lesigne, E. and Peign\'e, M.: 
Random walks in Weyl chambers and crystals. 
{\it Proc. Lond. Math. Soc.} {\bf104}, 323--358 (2012)

\bibitem{Li00}
Litvinchuk, G.:
{\it Solvability theory of boundary value problems and singular integral equations with shift.} 
Kluwer Academic Publishers, Dordrecht (2000)

\bibitem{PW}
Picardello, M. and Woess, W.:
{Martin boundaries of {C}artesian products of {M}arkov chains}.
{\it Nagoya Math. J.} {\bf 128}, 153--169 (1992)

\bibitem{Pr64}
Pruitt, W.: 
Eigenvalues of nonnegative matrices. 
{\it Annals of Math. Statistics} {\bf35}, 1797--1800 (1964)

\bibitem{Ra14}
Raschel, K.\ (with an appendix by Franceschi, S.):
Random walks in the quarter plane, discrete harmonic functions and conformal mappings.
{\it  Stochastic Process. Appl.} {\bf124}, 3147--3178 (2014)

\bibitem{Tutte}
Tutte, W: 
Chromatic sums for rooted planar triangulations. V. Special equations. 
{\it Canad. J. Math.} {\bf26}, 893--907 (1974)


\end{thebibliography}
%

\appendix

\section{Appendix: Explicit expressions for $w$}
\label{sec:expression_w}

It turns out that a suitable expression for $w$ has been found in \cite[Section 5.5.2]{FIM}. Let us briefly explain why such a function appears in \cite{FIM}. The main goal of \cite{FIM} is to develop a theory for solving a functional equation satisfied by the stationary probabilities generating function of reflected random walks in the quarter plane. The functional equation in \cite{FIM} is closed to ours (compare \cite[Equation (1.3.6)]{FIM} with~\eqref{eq:functional_equation}). Roughly speaking, the general solution of \cite{FIM} can be expressed as 
\begin{equation*}
     \int f(y)\frac{w'(y)}{w(x)-w(y)}\text{d}y
\end{equation*}
for some function $f$, with the same function $w$ as ours. Our situation is therefore simpler, since we can express the solutions directly in terms of $w$, without any integral.

In this section we simply state the expression of $w$, and we refer to \cite[Chapter 5]{FIM} or to \cite[Section 3]{KR11} for the details. The expression is
\begin{equation}
\label{eq:expression_w}
     w(x)=\frac{u(x_0)}{u(x)-u(x_0)}-\frac{u(x_0)}{u(0)-u(x_0)} 
\end{equation}
(the second term $\frac{u(x_0)}{u(0)-u(x_0)}$ in~\eqref{eq:expression_w} is to ensure that $w(0)=0$), where the function $u$ is different in the two cases $t\in(t_0,\infty)$ and $t=t_0$. 

\subsection{Case $t\in(t_0,\infty)$}
In that case, $u$ can be expressed in terms of Weiertrass elliptic functions, with the formula
\begin{equation}
\label{eq:expression_u}
     u(x)=\wp_{1,3}(s^{-1}(x)-\omega_2/2),
\end{equation}
where 
\begin{itemize}
\item $\wp_{1,3}$ is the Weierstrass elliptic function associated with the periods $\omega_1$ and $\omega_3$ defined in~\eqref{eq:expression_omega_1_2_3}, i.e., 
\begin{equation*}
     \wp_{1,3}(\omega) = \frac{1}{\omega^2}+\sum_{n_1,n_3\in{\bf Z}}\left\{ \frac{1}{(\omega-n_1\omega_1-n_3\omega_3)^2}-\frac{1}{(n_1\omega_1+n_3\omega_3)^2}\right\},
\end{equation*}
\item $\omega_1$ and $\omega_2$ are defined as below (with $\delta$ as in~\eqref{def_d}):
\begin{equation}
\label{eq:expression_omega_1_2_3}
     \omega_1= i\int_{x_1}^{x_2}\frac{\text{d}x}{\sqrt{-\delta(x)}},\qquad
     \omega_2 = \int_{x_2}^{x_3} \frac{\text{d}x}{\sqrt{\delta(x)}},\qquad
     \omega_3 = \int_{X(y_1)}^{x_1}\frac{\text{d}x}{\sqrt{\delta(x)}},
\end{equation}
\item $s(\omega)=g^{-1}(\wp_{1,2}(\omega))$, where $\wp_{1,2}$ is the Weierstrass elliptic function associated with the periods $\omega_1$ and $\omega_2$, and $g^{-1}$ is the reciprocal function of 
\begin{equation}
\label{eq:function_g}
g(x)=\left\{\begin{array}{lll}
          \displaystyle \frac{\delta''(x_{4})}{6}+\frac{\delta'(x_{4})}{x-x_{4}}& \text{if} & x_{4}\neq \infty,\vspace{1.5mm}\\
          \displaystyle \frac{\delta''(0)}{6}+\frac{\delta'''(0)x}{6} & \text{if} & x_{4}=\infty,\end{array}\right.
\end{equation}
\item $x_0\in (X(y_1),x_2)\setminus\{0\}$ is arbitrary.
\end{itemize}
   
\subsection{Case $t=t_0$}   
We have
\begin{equation}
\label{eq:expression_u_0}
     u(x)=\left(\frac{\pi}{\omega_3}\right)^{2}\left\{\sin\left(\frac{\pi}{\theta}\left[\arcsin\left(\frac{1}{\sqrt{\frac{1}{3}-\frac{2g(x)}{\delta''(1)}}}\right)-\frac{\pi}{2}\right]\right)^{-2}-\frac{1}{3}\right\},
\end{equation}
with $g$ as in~\eqref{eq:function_g} and 
  \begin{equation*}
          \theta=\arccos \left(-\frac{\sum_{-1\leq i,j\leq 1}i j p_{i,j}x_2^i y_2^j}
          {2\sqrt{\alpha(x_2)\widetilde \alpha(y_2)}}\right).
     \end{equation*}
     
\subsection{Remarks}
\label{sec:remarks}
It can be shown that:
\begin{itemize}
     \item The expressions given in~\eqref{eq:expression_u} and~\eqref{eq:expression_u_0} are a priori complicated, but it may happen that for some $\{p_{k,\ell}\}$, they become much simpler. If $p_{0,1}+p_{1,0}+p_{0,-1}+p_{-1,0}=1$ for instance, the function $u$ is rational. More generally, if $\omega_2/\omega_3\in{\bf Q}$, then $u$ is an algebraic function. See \cite[Proposition 15 and Remark 16]{KR11} for further remarks on~$u$.
     \item The function $u$ is is continuous w.r.t.\ the eigenvalue $t\in[t_0,\infty)$, see \cite[Section 2.2]{FR11}.
     \item At $t=t_0$ we have $x_2=x_3$ and $y_2=y_3$ (in fact $t_0=\inf\{t>0 : x_2 = x_3\}=\inf\{t>0 : y_2 = y_3\}$).
\end{itemize}

\end{document}